\documentclass[runningheads]{llncs}
\usepackage[T1]{fontenc}
\usepackage{amssymb}
\usepackage[margin=4.2cm,footskip=4.2cm]{geometry}
\usepackage{amsmath}
\usepackage{hyperref}

\newtheorem{thm}{Theorem}
\newtheorem{rem}{Remark}
\newtheorem{dfn}{Definition}
\newtheorem{cor}{Corollary}
\newcommand{\cF}{\mathcal{F}}
\newcommand{\cA}{\mathcal{A}}
\usepackage{graphicx}
\begin{document}
\title{Shape Theory $\&$ TDA via the Atiyah--Molino Reconstruction}
\author{No\'emie C. Combe\inst{1}\orcidID{0000-0003-4540-7376} \and
Hanna K. Nencka }
\authorrunning{N. Combe et al.}
\institute{University of Warsaw, Ul. Banacha 2, 02-097 Warsaw, Poland 
\email{n.combe@uw.edu.pl}\\
\url{https://noemie-combe-23.webself.net} \\}
\maketitle              % typeset the header of the contribution
\begin{abstract}

Reconstruction problems lie at the very heart of both mathematics and science, posing the enigmatic challenge: \emph{How does one resurrect a hidden structure from the shards of incomplete, fragmented, or distorted data?} In this paper, we introduce a new approach that harnesses the profound insights of the Vaisman Atiyah--Molino framework. In stark contrast to conventional methods that depend on persistent homology, our approach exploits the concept of the Vaisman centroid—an intrinsic invariant that encapsulates the averaged geometry of a data set—to resolve the inherent ambiguities of inverse problems.  In the present paper, we focus on the theory and applications of the Vaisman centroid, offering an innovative perspective for Topological Data Analysis that eschews persistent homology in favour of a unified geometric paradigm. The subsequent paper will extend these ideas to a full reconstruction scheme via the Atiyah–Molino framework. Our method not only provides a robust and computationally tractable framework for the recovery of hidden structures but also opens new avenues for the analysis of high-dimensional and noisy data across the mathematical sciences.


\end{abstract}
\section{Introduction}

Reconstruction problems have long captivated mathematicians and applied scientists alike, posing the profound challenge of recovering hidden structures from incomplete, noisy, or fragmented data. Such problems are ubiquitous, arising in disciplines as varied as cryo-electron microscopy, tomography, and inverse problems in imaging and structural analysis, see for instance \cite{CSSS,D,DGH}. Traditional approaches to these inverse problems often rely on iterative algorithms which, while effective in certain regimes, suffer from intrinsic limitations: they are typically computationally intensive, sensitive to noise, and frequently unable to resolve the inherent ambiguities present in underdetermined systems.

In this work, we introduce a new paradigm that {\it transcends} the conventional iterative schemes. Our approach unifies techniques from geometric integration and topological analysis within the pioneering frameworks developed by Vaisman\cite{V}, Neifeld \cite{N} and Atiyah--Molino \cite{A,M}. Rather than addressing each inverse problem in isolation, our approach employs a unified geometric language that recasts reconstruction as the extraction of hidden symmetries and invariants within the data. A longer version of the present work is available in \cite{CH}.
 
 By leveraging Vaisman’s deep insights into symmetry and its manifestation in invariant foliations, together with Neifeld’s analytic methodologies, we achieve a refined geometric decomposition. This decomposition, structured in terms of invariant foliations and moment maps, effectively isolates the ambiguities that naturally arise in reconstruction tasks, thereby ensuring both uniqueness and stability of the recovered solution.

Central to our method is the conceptual shift from traditional persistent homology and brute-force iteration toward a synthesis of geometric and topological invariants. The resulting framework is not only robust against noise but also computationally tractable, rendering it suitable for high-dimensional and large-scale applications. Our results demonstrate that by recasting the reconstruction problem in this unified geometric language, one can resolve the multiplicity of solutions and achieve error-bounded recovery in a manner that is both elegant and effective.

The {\bf Atiyah--Molino framework} reinterprets the reconstruction problem as a fiber bundle phenomenon. By splitting the problem into tangent directions (encoding local deformations, such as the tilting of a protein) and normal directions (capturing global invariants such as centroids and moments), this approach achieves significant computational tractability. The Atiyah--Molino exact sequence decomposes the problem into algebraic equations, thereby obviating the need for brute-force searches. Furthermore, the introduction of the Haantjes tensor---a curvature-like invariant---provides a precise measure of noise propagation, thereby allowing for error-bounded reconstructions.

This synthesis of differential geometry, algebraic topology, and integral analysis not only resolves long-standing ambiguities in inverse problems but also paves the way for transformative applications across science and engineering. In the subsequent sections, we develop the theoretical underpinnings of our approach and demonstrate its efficacy in several paradigmatic reconstruction scenarios.

Applications of our method include:

\begin{itemize}
\item {\bf Medicine}: Reconstructing tumours from sparse MRI slices with error guarantees.
\item {\bf Quantum Computing:} Inferring quantum states from noisy partial measurements.
\item {\bf Artificial Intelligence}: Training generative models to impute missing data through geometric manifold learning and differential topological constraints, as opposed to relying on conventional statistical priors or probabilistic inference frameworks.
\end{itemize}

%\begin{figure}[h]
%    \centering
 %   \includegraphics[width=0.8\textwidth]{projection2}
 %   \caption{Illustration of an imaging process: each projection image corresponds the line integrals of a molecule rotated by a
%three-dimensional rotation.}
   % \label{fig:1}
%\end{figure}

{\bf Acknowledgements} Both authors thank Philippe Combe for comments and discussions on this paper. This research is part of the project No. 2022/47/P/ST1/01177 co-founded by the National Science Centre and the European Union's Horizon 2020 research and innovation program, under the Marie Sklodowska Curie grant agreement No. 945339.For the purpose of Open Access, the author has applied a CC-BY public copyright licence to any Author Accepted Manuscript (AAM) version arising from this submission.
\includegraphics[width=1cm, height=0.5cm]{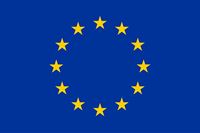}.

\section{Recollections on Vaisman--Neifeld's Setting}\label{S:VN}
Let us recall the setting from Vaisman--Neifeld \cite{V,N} and depicted in more detail in \cite{CH}. 
~

Let \( M \) be the space of smooth, non-symmetric three-dimensional varieties (geometric objects) embedded in \( \mathbb{CP}^3 \), and suppose that $M$ admits two independent rational projections \( \Pi_1, \Pi_2: M \dashrightarrow \mathbb{CP}^2 \). We impose the following hypothesis: 

\begin{itemize}
    \item \textbf{Algebraic structure:} The space \( M \) is defined over a commutative algebra \( \cA \), which intrinsically encodes the moment maps \( \mu_1, \mu_2 \) and dual connections \( \nabla_1, \nabla_2 \). 
    \item \textbf{Vaisman’s Bijection.} There exists a bijective correspondence between \( M \) and a foliated Riemannian manifold \( (\mathcal{F}, g) \), where the foliation \( \mathcal{F} = \{L_\alpha\} \) is smooth and \( g \) is a Riemannian metric defined on the leaves. This identification provides a geometric realization of the underlying algebraic structure.
\end{itemize}

\subsection{Foliation via Projection Equivalence}
For each rational projection \( \Pi_i \), one may define a foliation \( \cF_i \) on \( M \) by considering the fibers over the points of  $\mathbb{CP}^2$. Specifically, for every  \( p \in \mathbb{CP}^2 \), the corresponding leaf is given by 
\[
L_{i,p} = \{ O \in M \mid \Pi_i(O) = p \}.
\]
By invoking Vaisman's theorem,  one deduces that \( M \) decomposes into two mutually transverse foliations, \( \cF_1, \cF_2 \). Each foliation is endowed with a Riemannian metric \( g_i \) that is canonically induced by the Fubini--Study metric on \( \mathbb{CP}^2 \), thereby transferring the intrinsic geometric structure of the projective plane to the leaf space of $M$.

\subsection{Leafwise Metric and Connections}
Within Neifeld’s formalism, the connections \( \nabla_1, \nabla_2 \) restrict to Levi-Civita connections on the leaves \( L_{i,p} \). Moreover, the involution \( \iota \) --- characterised by the relation \( \Pi_2 = \iota \circ \Pi_1 \)---induces an isometry between the foliated Riemannian spaces \( (\cF_1, g_1) \) and \( (\cF_2, g_2) \). In particular, one has 
\[
\iota^* g_1 = g_2, \quad \text{and} \quad \iota^* \nabla_1 = \nabla_2,
\]
ensuring that both the metric and the associated connection structures are preserved under this involution. 

\subsection{Transverse Holonomy}
The well-posedness of the reconstruction problem is dictated by the holonomy groups associated with the foliations \( \cF_1, \cF_2 \). In the setting of non-symmetric geometric objects, the triviality of the holonomy allows for a global parallel transport of the moment maps $\mu_1$ and $\mu_2$ along the leaves, thereby reinforcing the uniqueness of the reconstruction. 

\section{Distributions, Haantjes Tensor Fields and Algebras}

The emergence of Haantjes tensor fields in this framework intimates a richer underlying structure governing the solution space. In our study, we probe the nature of the algebraic operations inherent to the reconstruction process---specifically, whether they adhere to {\it associativity} or deviate into {\it non-associativity}. This distinction is pivotal, as it bears significant consequences for the global coherence behaviour of the reconstructed entity.

\begin{dfn}
A {\bf distribution} $D$ on a smooth manifold $M$ is a smooth assignment of a linear subspace $D_p\subset T_pM$ to each point $p\in M$, where $T_pM$ is the tangent space at $p$. Formally, it is the subbundle of the tangent bundle $TM$.
\end{dfn}

~

A distribution is {\it integrable}  if through every point $p\in M$, there exists a submanifold $S\subset M$ such that $T_qS=D_p$ for all $q\in S$. 
By the {\it Frobenius theorem}, a distribution $D$ is integrable if and only if it is closed under the Lie bracket:
\[\forall X,Y\in \Gamma(D), \quad [X,Y]\in \Gamma(D).\]

\subsection{Haantjes tensor fields}

For vector fields \( X, Y \in \Gamma(D) \), the Haantjes tensor (restricted to $D$)  is defined as:
\[
H(X, Y) = \nabla_X Y - \nabla_Y X - [X, Y],
\]
where $\nabla$ is a connection on $TM$ and $[-,-]$ is the Lie bracket.

It quantifies the failure of \( D \) to integrate to a subfoliation. Specifically, if $\nabla$ preserves $D$ (i.e. $\nabla_XY\in\Gamma(D)$ for all $X\in\Gamma(TM), Y\in \Gamma(D)$) then the condition that $H=0$ implies that $D$ is integrable (Frobenius-compliant).

Indeed this can be proved by the fact that if $\nabla$ preserves $D$, then $H=0$ implies that $[X,Y]=\nabla_XY-\nabla_YX\in\Gamma(D)$, satisfying Frobenius' statement.

\subsection{Quantifying Non-Integrability}
\subsubsection{}\label{S:QNI} Let \( M \) be the foliated reconstruction space described in Sec. \ref{S:VN}, equipped with:
\begin{itemize}
    \item {\bf Transverse foliations} \( \mathcal{F}_1, \mathcal{F}_2 \) induced by projections \( \Pi_1, \Pi_2 \).
    \item A {\bf Haantjes tensor} \( H \in \Gamma(TM \otimes \Lambda^2 T^*M) \), measuring the non-integrability of the transverse distributions \( D_1, D_2 \) tangent to \( \mathcal{F}_1, \mathcal{F}_2 \). 
    
 \item An  {\bf  algebra } \( \mathcal{A} \) of parallel sections of \( \nabla_1 \otimes \nabla_2 \), with product \( \star \) defined by holonomy-corrected composition.
\end{itemize}

\subsection{Associativity, Quasigroups and Quasigroupoids} 
Recall that an algebra \( \mathcal{A} \) is {\bf associative} if:
\[
(a \star b) \star c = a \star (b \star c) \quad \text{for all } a, b, c \in \mathcal{A}. 
\]

\smallskip 

A set with a binary operation $\star$ where for all $a,b$, there exists unique $x,y$ such that $a\star x=b$ and $y\star a=b$. It is non-associative but satisfies the Latin square property.

\smallskip 

 A quasigroupoid is a category-like structure where morphisms between objects (in the sense of categories) are equipped with a partial binary operation $\star$, defined only for compatible pairs. 
For morphisms $f,g$, the product $f\star g$ exists if the codomain of $g$ matches the domain of $f$. Division is possible (as in quasigroups), but associativity is not required.

\begin{thm}
Suppose that \( M \) is defined as above (Sec. \ref{S:QNI})  and equipped with transverse foliations \( \mathcal{F}_1, \mathcal{F}_2 \), Haantjes tensors \( H_1, H_2 \) and the algebra $(\mathcal{A} , \star)$. Then,
\begin{enumerate}
    \item The algebra \( \mathcal{A} \) is {\bf associative} if and only if the Haantjes tensors \( H_1, H_2 \) vanish identically. Equivalently:
    \[
    \mathcal{A} \text{ is associative} \iff H_i = 0 \quad (i = 1, 2).
    \]
    \item If \( H_i \neq 0 \), the algebraic structure of \( \mathcal{A} \) is governed by a Moufang-like identity:
    \[
    (a \star b) \star (c \star a) = a \star (b \star c) \star a, \quad \forall a,b,c \in \mathcal{A}
    \]
    reflecting the curvature of \( \nabla_1 \otimes \nabla_2 \).
    \item \begin{itemize} 
    \item If the {\bf associative case} holds, then the reconstruction problem has a unique solution \( O \in M \), and \( M \) is globally biholomorphic to \( \mathbb{C}P^3 \). 
    \item If the {\bf non-associative case} holds, then solutions form a quasigroupoid under \( \star \), with non-unique reconstructions parameterized by the cohomology class \( [H_i] \in H^1(M, TM) \).
    \end{itemize}
\end{enumerate}
\end{thm}

\begin{proof}~
\begin{itemize}
\item[(1)] \textbf{Haantjes Tensor as an Obstruction.} The vanishing of the Haantjes tensor $H_i$	
ensures that the corresponding distributions $D_i$	
are integrable, thereby endowing the foliation 
$F_i$ with a Lie foliation structure. In this case, the induced connection $\nabla_i$	
 is necessarily flat, and the Leibniz rule forces the underlying algebra 
$\mathcal{A}$ to be associative. Reciprocally, if $\mathcal{A}$ is associative, the holonomy correction must be path-independent. Non-vanishing $H_i$ (torsion) introduces path-dependence in parallel transport via the torsion-twisted Lie bracket, causing non-associativity in $\star$. Thus the associativity forces $H_i=0$.
Therefore, \[\mathcal{A} \text{ is associative} \iff H_i = 0 \quad (i = 1, 2).\]
\item[]
\item[(2)] \textbf{Curvature and the Emergence of Non-Associativity.} The vanishing of $H_i$ ensures associativity by eliminating path-dependence. Conversely, when \(H_i \neq 0\), the curvature \(F_{\nabla_i}\) measures deviation from associativity. Indeed, the curvature  \(F_{\nabla_i}\) encodes the local holonomy and parallel transport deviations. For the tensor product connection  \(\nabla=\nabla_1 \otimes \nabla_2\), the curvature \(F_{\nabla}\) governs the failure of associativity in $\mathcal{A}$. Specifically, the associator $(a,b,c)=(a\star b)\star c-a\star (b\star c) $ is non-zero. The Bianchi identity for the tensor product connection \(\nabla_1 \otimes \nabla_2\), is a differential constraint on curvature: \[d^{\nabla}F_{\nabla}+[F_{\nabla}\wedge F_{\nabla}]=0,\] where $d^{\nabla}$ is the covariant exterior derivative. Under the imposed involution symmetry (an anti-automorphism $a\mapsto \overline{a}$ satisfying $ \overline{a\star b}= \overline{b}\star \overline{a}$), the Bianchi identity forces the associator to satisfy a symmetrized relation. Naturally, this symmetry reduces the general non-associativity to the Moufang identity: \[(a\star b)\star (c\star a)=a\star (b\star c)\star a, \] which is a weaker, symmetric form of associativity, thereby reflecting the inherent non-associative structure.
\item[]
\item[(3)] \begin{itemize}
\item \textbf{Global Implications.} In the associative setting, the triviality of holonomy guarantees a unique global parallel transport, thereby facilitating a coherent reconstruction. In contrast, non-associativity introduces nontrivial monodromy, which obstructs global uniqueness. In such circumstances, the quasigroupoid structure emerges from the Ehresmann connection defined on the product foliation \(\mathcal{F}_1 \times \mathcal{F}_2\).

\item \textbf{Quasigroupoid}.  The non-unique reconstructions (solutions) correspond to morphisms between objects, which are leaves of the foliations $\cF_1, \cF_2$.
The operation $\star$  is defined only for solutions sharing compatible projection data that is overlapping domains/codomains in the foliated space. The cohomology class $[H_i] \in H^1(M, TM)$ encodes deformations of the foliations, which act as automorphisms on the quasigroupoid structure. To summarise, the objects are leaves of $\cF_1, \cF_2$, morphisms correspond to reconstructions $O$, with compatible projections. The $\star$ operations correspond to composition of deformations along paths in $M$.
\end{itemize}
\end{itemize}
\end{proof}
\begin{rem}
The notion of quasigroupoid is  employed in this framework to underscore the following aspects: 
\begin{itemize}
\item {\bf Partial Operations:} The reconstruction process yields solutions that are  local, with composition operations defined only on restricted domains. 
\item {\bf  Categorical Structure:} The interplay between morphisms, representing the reconstruction mappings, and objects, corresponding to the leaves of the foliation, interact via a binary operation denoted as $\star$.
\item {\bf  Cohomological Dependence:} The cohomology class $[H_i]$ serves as a parameter space for deformations within this structure, analogous to the role of automorphisms in the classical groupoid settings. 
\end{itemize}
\end{rem}

\section{Atiyah--Molino Space and Reconstruction Criterion}
Let \( M \) be the \emph{Atiyah--Molino space}, associated with the reconstruction problem, defined as follows:

\begin{enumerate}
\item \( M \) is the space of {\bf smooth, non-symmetric} three dimensional varieties \( O \subset \mathbb{R}^3 \).
\item {\bf Foliation Structure:}
\( M \) is equipped with {\it two transverse foliations} \( \mathcal{F}_1, \mathcal{F}_2 \), where:
\[
\mathcal{F}_i: \quad O \sim O' \quad \text{if} \quad \Pi_i(O) = \Pi_i(O'),
\]
i.e., each leaf \( L_{i,p} \in \mathcal{F}_i \) consists of all varieties projecting to \( p \in \mathbb{R}^2 \) under \( \Pi_i \).
\item {\bf Atiyah--Molino Sequence:}
The tangent bundle \( TM \) splits as:
\[
0 \to T\mathcal{F}_1 \oplus T\mathcal{F}_2 \to TM \to NM \to 0,
\]
where \( NM \) is the normal bundle encoding transverse deformations.
\end{enumerate}

\subsection{Atiyah--Molino Reconstruction Space}
\begin{thm}
A three dimensional object \( O \in M \) is uniquely reconstructible from its projections \( \Pi_1(O) \) and \( \Pi_2(O) \) if and only if:
\begin{enumerate}
    \item The moment maps \( \mu_1 \) and \( \mu_2 \) are transverse sections of \( NM \), i.e.,
    \[
    d\mu_1 \wedge d\mu_2 \neq 0.
    \]
    \item The Haantjes tensor vanishes:
    \[
    H = 0.
    \]
\end{enumerate}

Moreover, when \( H = 0 \), the Atiyah--Molino sequence splits holonomy-free, and \( M \) is diffeomorphic to the total space of a trivial \( \mathbb{R}^2 \)-bundle over the leaf space \( M/\mathcal{F}_1 \times \mathcal{F}_2 \). The moment maps \( \mu_1 \) and \( \mu_2 \) provide a global trivialization:
\[
M \cong \mathbb{R}^2 \times \mathbb{R}^2, \quad O \mapsto \bigl( \mu_1(O),\, \mu_2(O) \bigr).
\]

\begin{itemize}
\item The vanishing of \( H \) ensures the algebra \( \mathcal{A} \) of parallel sections of \( NM \) is associative.
\item  In contrast, non-vanishing \( H \) induces a non-associative Moufang structure, thereby obstructing unique reconstruction.
\end{itemize}

\end{thm}

\begin{proof}~
\begin{itemize}
    \item \textbf{Transversality of Moment Maps:} The non-degeneracy condition 
    \[
    d\mu_1 \wedge d\mu_2 \neq 0
    \]
    ensures that \(\mu_1\) and \(\mu_2\) locally parametrize \(M\), effectively lifting the projections to coordinate functions.
    
    \item \textbf{Haantjes Tensor and Holonomy:} When   $H = 0,$ the foliations \(\mathcal{F}_1\) and \(\mathcal{F}_2\) are Lagrangian and integrable, which trivializes the reconstruction process. In contrast, if 
$ H \neq 0,$ monodromy is introduced, causing the parallel transport to depend on the path and leading to non-associativity.
    
    \item \textbf{Splitting of the Atiyah--Molino Sequence:} The trivialization 
    \[
    M \cong \mathbb{R}^2 \times \mathbb{R}^2
    \]
    follows from the Frobenius theorem when \(H = 0\), with \(\mu_1\) and \(\mu_2\) serving as Cartesian coordinates.
\end{itemize}
\end{proof}

\subsection{Deformation Theory under Non-Vanishing Haantjes Tensor}
\begin{cor}
Let \( (M, \mathcal{F}_1, \mathcal{F}_2, H) \) be an Atiyah--Molino reconstruction space equipped with a non-vanishing Haantjes tensor \( H \in \Gamma(TM \otimes \Lambda^2 T^*M) \). Then:
\begin{enumerate}
%\subsection*{1. Deformation Space}
\item The space of non-unique reconstructions forms a quasigroupoid \( Q \), where:

\begin{itemize}
    \item \textbf{Objects} are leaves of the foliations \( \mathcal{F}_1 \) and \( \mathcal{F}_2 \).
    \item \textbf{Morphisms} are deformations of solutions \( O \in M \), parameterized by the cohomology class \( [H] \in H^1(M, TM) \).
    \item \textbf{Partial Operation:} The composition \( \star \) is defined for deformations sharing compatible projection data, satisfying the Moufang identity:
    \[
    (a \star b) \star (c \star a) = a \star (b \star c) \star a.
    \]
\end{itemize}

%\subsection*{2. Cohomological Parameterization}
\item The first cohomology group \( H^1(M, TM) \) classifies infinitesimal deformations of the quasigroupoid \( Q \), with \( [H] \) acting as the obstruction class to:

\begin{itemize}
    \item integrability of the foliations \( \mathcal{F}_1 \) and \( \mathcal{F}_2 \).
    \item Associativity of the reconstruction algebra.
\end{itemize}
\end{enumerate}
The curvature $H$ induces a twisted Lie algebroid structure on $TM$, where the bracket $[-,-]_H$ deviates from the standard Lie bracket by terns proportional to $H$.
\end{cor}
\begin{proof}
Non-integrability (\( H \neq 0 \)) forces the leaves of \( \mathcal{F}_1 \) and \( \mathcal{F}_2 \) to intersect non-transversely, yielding multiple solutions. These solutions form a category with partial composition \( \star \), satisfying the Moufang identity due to the Bianchi identity for \( H \).

The class \( [H] \) parameterizes deformations via the Kodaira--Spencer map in deformation theory, where \( H^1(M, TM) \) governs first-order deformations of \( M \).

Finally, concerning the twisted Lie Algebroid, the bracket
\[
[X, Y]_H = [X, Y] + H(X, Y)
\]
defines a non-integrable algebroid, mirroring the non-associative algebra of \( Q \).
\end{proof}
\section{Conclusion}
\subsection{Applications} We suggest some possible applications of our method to generative models. 
\paragraph{Goal and Key idea.} The goal is to train generative models to impute missing data (e.g., incomplete images, sparse sensor readings) using differential topology rather than statistical priors. Here we propose to model the data manifold $M$ using transverse foliations $\mathcal{F}_1, \mathcal{F}_2$, where:
\begin{itemize}
    \item $\mathcal{F}_1$: Foliation of observed features (e.g., visible pixels),
    \item $\mathcal{F}_2$: Foliation of missing or masked features.
\end{itemize}
The reconstruction is governed by:
\begin{itemize}
    \item the Haantjes tensor $H_i$,
    \item the curvature of a chosen connection $F_\nabla$, and
    \item the reconstruction algebra $(\mathcal{A}, \star)$.
\end{itemize}

\paragraph{Data Embedding.}
Given a dataset $D \subset \mathbb{R}^n$, learn an embedding via an autoencoder:
\[
\varphi: D \to M \subset \mathbb{R}^d, \quad (d \ll n),
\]
where $M$ approximates the intrinsic data manifold.
Define a pair of transverse foliations:
\[
\mathcal{F}_1: \text{leaves along observed features}, \quad \mathcal{F}_2: \text{leaves along missing features}.
\]

\paragraph{Connection $\nabla$.}
Define a connection on the tensor bundle $T\mathcal{F}_1 \otimes T\mathcal{F}_2$ that preserves the geometric structure of $M$, e.g., a Levi-Civita connection in the Riemannian case.

\paragraph{Algebra $\mathcal{A}$.}
Let
\[
\mathcal{A} = \Gamma^\parallel(T\mathcal{F}_1 \otimes T\mathcal{F}_2),
\]
denote the space of parallel sections under $\nabla$.

For $a, b \in \mathcal{A}$ representing data patches,
\[
(a \star b)(p) := \text{holonomy-corrected composition at } p.
\]

\paragraph{Input:} Partial observation $x_{\mathrm{obs}}$ (a leaf of $\mathcal{F}_1$).

\paragraph{Output:} Full reconstruction $x_{\mathrm{full}} \in M$.

\paragraph{Process:}
\[
x_{\mathrm{full}} = x_{\mathrm{obs}} \star g, \quad g \in \mathcal{A} \text{ (parallel section along } \mathcal{F}_2).
\]

We now discuss the role of the Haantjes Tensor in this setting. 

\paragraph{Uniqueness Criterion.}
\begin{itemize}
    \item If $H^i = 0$: $\mathcal{F}_i$ is integrable $\Rightarrow$ unique imputation via associative algebra.
    \item If $H^i \neq 0$: $\mathcal{F}_i$ is non-integrable $\Rightarrow$ multiple imputations (quasigroupoid structure).
\end{itemize}

\subsection{Final Word}
The problem of reconstructing three-dimensional geometrical objects from planar projections is a rich mathematical challenge that draws from integration theory, differential geometry, and algebraic structures. Using methods pioneered by Gelfand, and informed by the work of Neifeld--Vaisman, we establish connections between integral geometry, curvature tensors, and foliations. This approach has enables us to redefine reconstruction via the theory of Atiyah--Molino, providing new insights into the nature of inverse problems in medical imaging and beyond, and solving important issues following from traditional methods.


\end{document}